\documentclass[a4paper,12pt]{amsart}
 
\usepackage{amssymb, amsmath, amsthm, amsfonts}
\usepackage{extsizes}
\usepackage{hyperref}
\usepackage{xcolor,graphicx}

\theoremstyle{plain}
\newtheorem{theorem}{Theorem}[section]

\newtheorem{corollary}[theorem]{Corollary}
\newtheorem{proposition}[theorem]{Proposition}

\theoremstyle{definition}
\newtheorem{definition}[theorem]{Definition}
\newtheorem{remark}[theorem]{Remark}

\newtheorem*{problem}{Problem} 

 \newcommand{\bfr}{\mathbf{r}}
\newcommand{\bfa}{\mathbf{a}}
\newcommand{\bft}{\mathbf{t}}
\newcommand{\bfn}{\mathbf{n}}
\newcommand{\bfb}{\mathbf{b}}

 \def\cy{\mathcal C}
\def\r{\mathbb R}

\usepackage{xcolor,graphicx}

\begin{document}

\title[Cylindrical curves in terms of curvature and torsion]{A necessary condition for cylindrical curves in terms of curvature and torsion}

\author[R. L\'opez]{Rafael L\'opez}
\address{Rafael L\'opez \\ Department of Geometry and Topology\\ University of Granada. 18071 Granada.   Spain}
\email{rcamino@ugr.es}

\begin{abstract}
 We establish necessary conditions for a regular curve to lie on a circular cylinder in terms of its curvature $\kappa$ and torsion $\tau$. By identifying a fundamental function $\psi = \sin^2 \alpha$, representing the squared sine of the angle between the tangent vector and the axis of the cylinder,  we reduce the geometric inclusion problem to a compatibility condition between  an explicit eighth-degree polynomial equation  and a   differential equation for $\psi$. This approach yields a single  ODE involving only $\kappa$ and $\tau$ that governs the inclusion of the curve in the cylinder.     The robustness of  this framework is   demonstrated through specific examples of cylindrical curves. Furthermore, we analyze the case of curves with constant curvature $\kappa_0$, obtaining an explicit ODE for  the torsion.   Remarkably, we prove that if $\kappa_0 = 1/\rho$, this equation admits   an explicit, exact solution for $\tau$.\end{abstract}

\maketitle

\section{Introduction and statement of the problem}\label{sec1}

In this paper, we investigate the following question:

\begin{problem}
Given a circular cylinder $\cy$ in Euclidean space $\r^3$, how can one characterize a regular curve $\bfr\colon I\subset\r\to\r^3$ lying on $\cy$  in terms of its curvature $\kappa$ and   torsion $\tau$? 
\end{problem}
 
An answer to this problem has   interesting  applications  in computer vision  \cite{fo,ij,jt,krsw,kp}.   Suppose that we know  the  trajectory of a point in $\r^3$ by tracking its coordinates at discrete positions. From this information, we can determine   the curvature and the torsion of the curve that describes the path of these points. The question to be answered is whether such a set of  points is contained in  some circular cylinder in $\r^3$, or in a circular cylinder with a specific radius.  The main challenge is to deduce whether a curve lies on  a two-dimensional surface solely from  its one-dimensional  intrinsic invariants $\kappa$ and $\tau$. In general, a regular curve $\bfr\colon I\subset\r\to\r^3$ that is  not a straight line  has a well-defined   curvature $\kappa>0$ and   torsion $\tau$. Both geometric quantities  uniquely determine  the curve $\bfr$ up to a rigid motion of $\r^3$.

Although the formulation of this problem is simple, its solution, if any,  is surprisingly complex.  
If we replace the cylinder $\cy$  with another surface $S$, the problem has been fully solved for only two specific surfaces: the plane and the  sphere. Indeed, if $S$ is a plane of $\r^3$, then a curve $\bfr$ lies   in $S$ if and only if its torsion $\tau$ vanishes identically. On the other hand, if $S$ is    a  sphere   of radius $r>0$, then   a non-planar curve $\bfr$ is contained in $S$ if and only if  
\begin{equation}\label{es}
\frac{1}{\kappa^2}+\frac{\kappa'^2}{\kappa^4\tau^2}=r^2,
\end{equation}
where the prime $(')$ stands for the derivative with respect to the arc length. This is a first-order ODE   in $\kappa$. If the radius of the sphere is unknown, then   $\bfr$ is  contained in some sphere of $\r^3$ if and only if $\kappa$ and $\tau$ satisfy the second-order ODE in $\kappa$ given by 
\begin{equation}\label{es2}
\frac{\tau}{\kappa}-\left(\frac{\kappa'}{\kappa^2\tau}\right)'=0.
\end{equation}
The proofs of \eqref{es} and \eqref{es2} appear in standard  textbooks such as \cite{dc,ku,st}.  For a rigorous global formulation and an explicit integration of these spherical conditions, we refer to the classical works of Wong \cite{wo1,wo2}. Extensions of \eqref{es} and \eqref{es2} to other ambient spaces   for curves contained in totally geodesic and umbilical surfaces have been recently developed    \cite{da,dd,lo1,pp,pm1}.

We emphasize that our problem fundamentally differs from the study of curves on a surface $S$ via their geodesic curvature and geodesic torsion. Unlike the 
standard curvature $\kappa$ and torsion $\tau$, which are intrinsic invariants of the curve itself, 
the geodesic quantities depend on the 
specific embedding of the curve within the surface. 

Returning to the initial problem regarding circular cylinders, few results have been obtained in the literature \cite{krsw,ko,sh}. 
 In \cite{krsw}, the authors  attempted to characterize cylindrical containment by computing the resultants of the differential invariants. However, their method typically leads to a resultant of degree $12$ (or higher), resulting in a symbolic expression of such overwhelming complexity that it becomes analytically unmanageable.   
 The problem was also addressed in the early work of Ko \cite{ko},  who established necessary and sufficient conditions by formulating a system of coupled differential-algebraic equations that depends on auxiliary angular variables. This approach does not achieve the complete algebraic elimination of these variables into a single, explicit polynomial in terms of the intrinsic invariants $\kappa$ and $\tau$.

More recently, Starostin and van der Heijden \cite{sh} employed a kinematic approach based on moving frames, relating the intrinsic invariants to an angle $\xi$ (measuring the rotation between the Frenet frame and a secondary frame), which also emerges as the root of an eighth-degree polynomial in $\kappa, \tau$, and $\kappa'$.

The main result of this paper is presented in Section \ref{sec2}. We prove that if a regular curve lies on a circular cylinder of radius $\rho$, then $\kappa$ and $\tau$ satisfy a specific ordinary differential equation (Theorem \ref{t1}). This equation is obtained as follows. Let $\alpha$ be the angle that the curve makes with the axis of 
the cylinder, and let $\psi = \sin^2 \alpha$. We prove that $\psi$ must simultaneously satisfy an $8$th-degree polynomial equation $\mathcal{P}(\psi) = 0$ and a differential equation $\mathcal{Q}(\psi, \psi') = 0$. 
The compatibility condition between these two equations, obtained by eliminating $\psi$, yields a single  ODE, $\mathcal{E}(\rho, \kappa, \kappa', \kappa'', \tau, \tau') = 0$. This resultant equation represents 
the fundamental constraint that the curvature and torsion must satisfy for the curve to be contained in the cylinder.

In Section \ref{s3}, we show how our result can be utilized to establish necessary conditions for curves with specific curvature or torsion profiles to lie on a circular cylinder. A particularly remarkable case arises when the curvature is assumed to be a constant $\kappa_0$.  In this scenario, the general ODE $\mathcal{E}=0$ is explicitly obtained (Theorem \ref{t2}).   We prove that, under the     condition $\kappa_0 = 1/\rho$, this equation   admits a completely explicit, exact   solution for the torsion (Corollary \ref{cor34}). Finally, in Section \ref{s4}, we illustrate the versatility of our method by applying it to classical examples, such as Viviani's curve.

\section{Characterization via the Frenet frame}\label{sec2}

In this section, we establish the main result of the paper (Theorem \ref{t1}). The strategy of the proof is constructive: we will use the Frenet equations to algebraically eliminate the direction cosines of the axis of the cylinder,  reducing the geometric condition of being a cylindrical curve to an ordinary differential equation and a polynomial equation of degree $8$.

Let $\cy\subset\r^3$ be a circular cylinder of radius $\rho>0$   whose axis has a unit direction vector $\bfa$. The cylinder is given by
\begin{equation}\label{eq00}
    \cy = \{\mathbf{x}\in\r^3\colon \|\mathbf{x}\|^2-  \langle \mathbf{x},\bfa\rangle^2 =  \rho^2\}.
\end{equation}

 Let $\bfr$ be a regular curve lying on  $\cy$, $\bfr\colon I\subset\r\to\cy$, $\bfr=\bfr(s)$, where $s$ is the arc length parameter. 
Let $\{\bft,\bfn,\bfb\}$ be the Frenet frame of $\bfr$. Let $\kappa$ and $\tau$ be its curvature and torsion, respectively. Throughout this section, we assume that the curve 
is biregular, that is, $\kappa > 0$ and $\tau \neq 0$ everywhere on $I$, to 
ensure that the Frenet frame is well-defined and to avoid singularities 
in our computations.  The Frenet equations are 
\begin{equation*}
    \begin{split}
        \bft'&=\kappa\bfn\\
        \bfn'&=-\kappa\bft+\tau \bfb,\\
        \bfb'&=-\tau\bfn.
    \end{split}
\end{equation*}
Let us express the axis $\bfa$ as a linear combination of $\{\bft,\bfn,\bfb\}$ using its direction cosines, 
\begin{equation}\label{eqa}
    \bfa = \cos\alpha\,\bft + \cos\beta\,\bfn + \cos\gamma\,\bfb.
\end{equation}
By the Frenet equations, the derivatives of the direction cosines are
 \begin{equation}\label{ss}
 \begin{split}
     (\cos\alpha)'&=\kappa\cos\beta,\\
    ( \cos\beta)'&=-\kappa\cos\alpha+\tau\cos\gamma,\\
     (\cos\gamma)'&=-\tau\cos\beta.
 \end{split}
 \end{equation}

Differentiating  the equation 
\begin{equation}\label{eq0}
\|\bfr\|^2-\langle\bfr,\bfa\rangle^2 = \rho^2
\end{equation}
with respect to $s$, we  deduce that
\begin{equation}\label{eq1}
    \langle\bft,\bfr\rangle-\langle\bfr,\bfa\rangle\langle\bft,\bfa\rangle = 0, 
\end{equation}
that is,
$$\langle \bft,\bfr-\langle\bfr,\bfa\rangle\bfa\rangle=0.$$
This implies that $\bfr-\langle\bfr,\bfa\rangle\bfa$ is a linear combination of $\bfn$ and $\bfb$. Thus, there are two functions of $s$, $A(s)$ and $B(s)$, such that 
\begin{equation}\label{eq2}
\bfr-\langle\bfr,\bfa\rangle\bfa=A\bfn+B\bfb.
\end{equation}
\begin{remark}
While the cylindrical condition \eqref{eq2} is formally identical to the   spherical relation $\mathbf{r}-\mathbf{c}_0 = \alpha \mathbf{n} + \beta \mathbf{b}$, a fundamental geometric difference makes the cylindrical problem significantly harder. For a sphere, the center $\mathbf{c}_0$ is a fixed point ($\mathbf{c}_0' = \mathbf{0}$), yielding the relatively simple ODE \eqref{es2}.   In contrast, the term  $\langle\bfr,\bfa\rangle\bfa$ in \eqref{eq2} is not constant.  This introduces an additional degree of freedom, which is precisely why the cylindrical characterization escalates into a   highly  difficult problem. 
\end{remark}

Multiplying \eqref{eq2} by $\bfn$ and $\bfb$, and using \eqref{eqa}, we obtain the following expressions for $A$ and $B$, 
\begin{equation}\label{eq3}
\begin{split}
A&=\langle \bfr-\langle\bfr,\bfa\rangle \bfa,\bfn\rangle=\langle\bfr,\bfn\rangle-\cos\beta\langle\bfr,\bfa\rangle,\\
B&=\langle \bfr-\langle\bfr,\bfa\rangle \bfa,\bfb\rangle=\langle\bfr,\bfb\rangle-\cos\gamma\langle\bfr,\bfa\rangle.
\end{split}
\end{equation}

 By using \eqref{eq1} and the definitions of $A$ and $B$ in \eqref{eq3}, we obtain the derivatives of $A$ and $B$, namely,  
\begin{eqnarray}
 A' &=&  -\kappa\langle\bfr,\bft\rangle + \tau \langle\bfr,\bfb\rangle -\cos\alpha\cos\beta -(-\kappa\cos\alpha+\tau\cos\gamma)\langle\bfr,\bfa\rangle\nonumber\\
 &=&\tau B -\cos\alpha\cos\beta.\label{ab1}\\
 B'& =&  -\tau \langle\bfr ,\bfn\rangle - \cos\gamma\cos\alpha+\tau\cos\beta \langle\bfr,\bfa\rangle\nonumber	\\	
 &=&-\tau A- \cos\gamma\cos\alpha.\label{ab2}
\end{eqnarray}

Differentiating \eqref{eq1}, and using \eqref{eqa} again,  yields
\[    0  = 1+ \kappa \langle\bfn,\bfr\rangle-\cos^2\alpha + \langle \bfr, \bfa\rangle\kappa \cos\alpha= \kappa A + 1 - \cos^2\alpha.
\]
This gives 
\begin{equation}\label{eq4}
    A = - \frac{\sin^2\alpha}{\kappa}.
\end{equation}
 
 From \eqref{ab1}, we obtain $B$, namely,
 \begin{equation}\label{ab3}
 B=\frac{1}{\tau}\left(A'+\cos\alpha\cos\beta\right).
 \end{equation}
 We now replace the expressions for $A$ and $B$   into  \eqref{ab2}, obtaining  
 $$\frac{d}{ds}\left(\frac{1}{\tau}\left(A'+\cos\alpha\cos\beta\right)\right)+\tau A+\cos\gamma\cos\alpha=0.$$
 Equivalently, 
 \begin{equation}\label{ab33}
 -\frac{d}{ds}\left(\frac{1}{\tau} \frac{d}{ds}\left(\frac{\sin^2\alpha}{\kappa}\right)\right)+\frac{d}{ds}\left(\frac{\cos\alpha\cos\beta}{\tau} \right)-\tau \frac{\sin^2\alpha}{\kappa}+\cos\gamma\cos\alpha=0.
 \end{equation}
In this equation, all three direction cosines appear.   The next step is to express this equation in terms of only one direction cosine, specifically, of  $\cos\alpha$.

First, using the fact that $\bfa$ is a unit vector, we can write 
\begin{equation}\label{cog}
\cos\gamma = \pm\sqrt{1-\cos^2\alpha-\cos^2\beta}.
\end{equation}
 The local choice of the sign depends on the orientation of the curve, but this ambiguity will not affect our subsequent computations since the equation will be eventually squared.

\begin{remark}
Let us observe that  combining     equations   \eqref{ab1} and \eqref{ab2} does not yield any further information. Indeed, from both equations, we obtain 
$$-\cos\alpha(A\cos\beta+B\cos\gamma)=AA'+BB'=\frac12(A^2+B^2)'=\frac12(\rho^2)'=0,$$
because $\rho$ is constant. However, the left-hand side of this identity is trivially $0$ by the definitions of $A$ and $B$ in \eqref{eq2} together with \eqref{ss}.
\end{remark}

From \eqref{eq2}, we have $\rho^2=A^2+B^2$ and using the values of $A$ and $B$ in \eqref{eq4} and \eqref{ab3}, respectively, we obtain 
$$\rho^2=A^2+B^2=\left( \frac{\sin^2\alpha}{\kappa}\right)^2+\frac{1}{\tau^2}\left(\frac{d}{ds} \left(\frac{\sin^2\alpha}{\kappa}\right)-
\cos\alpha\cos\beta\right)^2.$$
From the first equation of \eqref{ss}, we have 
\begin{equation}\label{cob}
\cos\beta=\frac{(\cos\alpha)'}{\kappa}.
\end{equation}
 Thus, the above identity becomes
\begin{equation}\label{r7}
\rho^2=  \frac{\sin^4\alpha}{\kappa^2} +\frac{1}{\tau^2}\left[\frac{d}{ds} \left(\frac{\sin^2\alpha}{\kappa}\right)+\frac{\alpha'\sin(2\alpha)}{2\kappa}  \right]^2.
\end{equation}

We introduce the main variable that will govern the rest of our study.

\begin{definition}
 We define the \emph{fundamental function} $\psi$ as
\begin{equation}\label{def}
\psi = \sin^2\alpha.
\end{equation}
\end{definition}

Notice that $\psi$ cannot be identically  $0$ because in such a case, $\langle\bft,\bfa\rangle^2=1$ and $\bfr$ is a straight line.  Analogously, if $\psi$ is identically $1$, then  $\langle\bft,\bfa\rangle=0$ for all $\in I$. Then $\bfr$ is a planar curve, and $\tau$ would be $0$, which was initially discarded.

Using the function $\psi$, equation \eqref{r7} becomes 
\begin{equation}\label{eq5}
\rho^2=  \frac{\psi^2}{\kappa^2} +\frac{1}{\tau^2}\left[  \frac{3\psi'}{2\kappa}-\frac{\psi\kappa'}{\kappa^2}   \right]^2.
\end{equation}
In particular, $\psi^2\leq \rho^2\kappa^2$. This gives
\begin{equation}\label{p1}
\psi'=\frac{2}{3} \left(\frac{\psi  \kappa '}{\kappa }\pm\tau  \sqrt{\rho ^2 \kappa ^2-\psi ^2}\right).
\end{equation}
The presence of the $\pm$ sign arises from taking the square root in \eqref{eq5}. Equation \eqref{p1} (or \eqref{eq5}) is  the ODE $\mathcal{Q}=0$ mentioned in the Introduction.  For the subsequent algebraic manipulation, this choice of sign is not critical, as we will eventually square the expressions to eliminate the radical.

Let
$$P=\sqrt{\rho ^2 \kappa ^2-\psi ^2}.$$
Differentiating \eqref{p1}, and using the expression for $\psi'$ in \eqref{p1} yields
\begin{equation}\label{p2}
\begin{split}
\psi''&=\frac{2 \left(5 \rho ^2 \kappa ^3 \tau  \kappa '-\psi  \kappa '^2 P-\kappa  \psi  \left(4 \tau  \psi  \kappa '-3 \kappa '' P\right)-\kappa ^2 \psi  \left(2 \tau ^2 P+3 \psi  \tau '\right)+3 \rho ^2 \kappa ^4 \tau '\right)}{9 \kappa ^2 P}.
\end{split}
\end{equation}

We write Eq. \eqref{ab33} in terms of the function $\psi$. From \eqref{ss} and the definition of $\psi$ in \eqref{def}, we have 
\begin{equation*}
\begin{split}
\cos\beta&=\frac{(\cos\alpha)'}{\kappa}=-\frac{\psi'}{2\kappa\sqrt{1-\psi}},\\
\cos\gamma&=\pm \sqrt{1-\cos^2\alpha-\cos^2\beta}=\pm\frac{\sqrt{4\kappa^2\psi(1-\psi)-\psi'^2}}{2\kappa\sqrt{1-\psi}}.
\end{split}
\end{equation*}

Then \eqref{ab33} becomes
\begin{equation*}
\begin{split}
0&=-4 \tau  \psi  \kappa '^2+\kappa  \left(5 \tau  \kappa ' \psi '+2\psi  \left( \tau  \kappa ''- \kappa ' \tau '\right)\right)\\
&+\kappa ^2 \left( \pm \tau ^2 \sqrt{4 \kappa ^2 (1-\psi ) \psi -\psi '^2}+3 \tau ' \psi '-3 \tau  \psi ''-2 \tau ^3 \psi \right).
\end{split}
\end{equation*}
In this equation, we substitute the values of $\psi'$ and $\psi''$ given in \eqref{p1} and \eqref{p2}, respectively. Isolating the term with the radical, we obtain
$$
\frac{\psi ^2 \kappa '}{P}+\kappa\tau\psi = \pm \kappa \sqrt{-\frac{2 P \tau  \psi  \kappa '}{\kappa }-\frac{\psi ^2 \kappa '^2}{\kappa ^2}-\kappa ^2 \left(\rho ^2 \tau ^2+9 (\psi -1) \psi \right)+\tau ^2 \psi ^2}.$$
By squaring both sides, the ambiguity of the $\pm$ sign is naturally eliminated. Simplifying, we arrive first at 
\begin{equation*}
0=2 \sqrt{\rho^2\kappa^2-\psi^2} \rho ^2 \kappa  \tau  \psi  \kappa '+\rho ^2 \psi ^2 \kappa '^2+\kappa ^2 \left(\rho ^2 \kappa ^2-\psi ^2\right) \left(\rho ^2 \tau ^2+9 (\psi -1) \psi \right).
\end{equation*}
and, finally,  squaring   to eliminate the radical, we obtain

\begin{equation}\label{eq7}
0=4 \rho^4 \kappa^2 \tau^2 \psi^2 \kappa'^2 (\rho^2 \kappa^2 - \psi^2) - \left[ \rho^2 \psi^2 \kappa'^2 + \kappa^2 (\rho^2 \kappa^2 - \psi^2) (\rho^2 \tau^2 + 9 \psi (\psi - 1)) \right]^2.
\end{equation}
This is a polynomial equation of degree $8$ in $\psi$ which can be written as 
$$0=\mathcal{P}(\psi)=\sum_{n=0}^8 P_n\psi^n,$$
where the coefficients $P_n$   depend only on $\kappa$, $\tau$, and $\kappa'$, without involving higher-order derivatives.  When we put $\psi$ into \eqref{p1}, then it will appear  $\tau'$ and $\kappa''$. This ODE is the desired equation involving only $\kappa$ and $\tau$. Thus, the previous arguments have shown the following result.

 \begin{theorem}\label{t1}
 Let $\mathbf{r}(s)$ be a regular curve parametrized by arc length with curvature $\kappa(s) > 0$ and 
 non-zero torsion $\tau(s)$. If the curve lies on a circular cylinder of radius $\rho$, 
 then $\kappa$ and $\tau$ must satisfy an ordinary differential equation 
 \begin{equation}\label{e1}
 \mathcal{E}=\mathcal{E}(\rho,\kappa,\kappa',\kappa'',\tau,\tau')=0,\end{equation}
 obtained by 
 substituting the real roots $\psi$ of the eighth-degree polynomial equation 
 \eqref{eq7} into the differential equation \eqref{p1}.
 \end{theorem}

 \begin{remark}
It is worth noting that the final ODE $\mathcal{E} = 0$ is obtained by eliminating $\psi$ between equations $\mathcal{P} = 0$ and $\mathcal{Q} = 0$. Conceptually, if $\{\psi_1, \dots, \psi_8\}$ are the roots of the polynomial equation, the condition $\mathcal{E} = 0$ can be understood as the vanishing of the product:
\begin{equation*}
\prod_{i=1}^{8} \mathcal{Q}(\psi_i, \psi_i') = 0.
\end{equation*}
This ensures that the resulting ODE is a single, well-defined expression that is independent of the specific root chosen, reflecting the symmetric nature of the elimination process.
\end{remark}

From an applied perspective, such as in 3D shape recognition or spatial curve tracking, reducing the geometric problem to the polynomial equation \eqref{eq7} is theoretically highly advantageous. For a given point on 
a trajectory, the coefficients evaluate to real scalars, transforming the characterization into a one-dimensional root-finding task. However, it is important to acknowledge the numerical challenges, especially in extracting   the derivatives of $\kappa$ and $\tau$ in order to manage and numerically solve the ODE $\mathcal{E}=0$.

\begin{corollary} Let $\bfr$ be a curve with curvature $\kappa>0$ and torsion $\tau\not=0$. If $\kappa$ and $\tau$ satisfy the differential equation \eqref{p1}, then the curve $\tilde{\bfr}$ whose curvature is $\kappa$ and torsion $-\tau$ also satisfies \eqref{p1}.
\end{corollary}

\begin{proof}
First, observe that the polynomial equation \eqref{eq7}, which restricts the values of $\psi$, depends on the torsion only through its square, $\tau^2$. Therefore, substituting $-\tau$ for $\tau$ leaves the equation invariant, meaning that both curves share the same valid functions for $\psi$.

Second, replacing $\tau$ with $-\tau$ in the differential equation \eqref{p1} simply changes the $\pm$ sign to $\mp$. Since the equation inherently accounts for both signs, the relation holds for $-\tau$ whenever it holds for $\tau$. 
\end{proof}
 
Geometrically, the curve $\tilde{\bfr}$ corresponds to a reflection of $\bfr$ across a plane, a rigid transformation that naturally preserves its inclusion in a circular cylinder of the same radius.

It is worth noting that Theorem \ref{t1} provides a necessary condition for a curve to lie on a circular cylinder. The converse problem, establishing sufficiency, is significantly more delicate. Since our constructive proof relies on differentiating the geometric constraints and squaring equations to eliminate radicals, the reverse process would inherently introduce arbitrary constants of integration.

It is worth making two further observations.

\begin{remark}
 In contrast to previous approaches, our method is strictly algebraic 
and geometrically direct because we focus on the  function $\psi = \sin^2\alpha$, 
which represents the inclination of the tangent vector relative to the cylinder's fixed axis. 
This allows us to explicitly eliminate the direction cosines 
without invoking secondary curves. 
 
\end{remark}

\begin{remark}
The condition \eqref{es} for a regular curve to be spherical can be recovered as a limiting case of Theorem \ref{t1}. Specifically, if a curve is contained in a sphere of radius $\rho$, the cylinder axis $\mathbf{a}$ is no longer uniquely defined, and we may consider that $\bfa$ reduces to a point. In particular, in the limiting case we have $\alpha = \pi/2$. Consequently, $\psi \equiv 1$ identically. Substituting this into equation \eqref{eq7}, the expression reduces to:
$$\frac{\kappa'^2}{\kappa^2 \tau^2} = \rho^2 \kappa^2 - 1,$$
which coincides with the classical spherical condition \eqref{es}. Alternatively, applying $\psi \equiv 1$ (and thus $\psi' = 0$) to \eqref{eq5}, we obtain:
$$\rho^2 = \frac{1}{\kappa^2} + \frac{1}{\tau^2} \frac{\kappa'^2}{\kappa^4},$$
which again recovers \eqref{es}.
\end{remark}

\section{Special classes of cylindrical curves}\label{s3}

In this section, we apply the general characterization established in Theorem \ref{t1} to specific families of spatial curves. We first verify that our results consistently recover the classical geometry of circular helices. Subsequently, we explore the more complex cases of generalized helices and curves with constant curvature.

\subsection{Circular helices}
The most elementary example of a cylindrical curve is the circular helix.  A circular helix is characterized by the fact that both $\kappa$ and $\tau$ are constant. The following result is known, but we recover it from Theorem \ref{t1}.

\begin{proposition}\label{prop3.1}
Let $\mathbf{r}$ be a curve with constant curvature $\kappa_0 > 0$ and constant torsion $\tau_0\not=0$. If $\mathbf{r}$ lies on a circular cylinder $\mathcal{C}$ of radius $\rho$, then   
\begin{equation}\label{pk}
\rho = \frac{\kappa_0}{\kappa_0^2 + \tau_0^2}.
\end{equation}
\end{proposition}

\begin{proof}
If $\kappa_0$ and $\tau_0$ are constants, then all coefficients in the polynomial equation \eqref{eq7} are constant; hence, the function $\psi$ must be a constant root. From    \eqref{eq5}, we have $\psi = \rho\kappa_0$. Identities \eqref{ss} imply that $\cos\beta=0$ and $\cos\gamma$ is constant. Substituting these  into   \eqref{ab33}, we obtain
 $$\psi = \frac{\kappa_0^2}{\kappa_0^2 + \tau_0^2}.$$
  Finally, using the relation $\rho = \psi/\kappa_0$, we get \eqref{pk}. 
\end{proof}

\subsection{Generalized helices (Lancret curves)}
A natural generalization of the circular helix is the Lancret curve (or generalized helix), characterized by the fact that the ratio $\tau/\kappa = a$ is constant. For such curves, our Theorem \ref{t1} provides a specific necessary condition for cylindrical inclusion. By substituting $\tau = a\kappa$ into \eqref{p1} and \eqref{eq7}, these  necessary conditions become  
\begin{equation}\label{l1}
\psi'=\frac{2}{3} \left(\frac{\psi  \kappa '}{\kappa }\pm a\kappa \sqrt{\rho ^2 \kappa ^2-\psi ^2}\right),
\end{equation}
and
\begin{equation}\label{l2}
\pm 2 a \rho^2 \kappa^2 \psi \kappa'\sqrt{\rho^2 \kappa^2 - \psi^2} =  \rho^2 \psi^2 \kappa'^2 + \kappa^2 (\rho^2 \kappa^2 - \psi^2) (a^2 \rho^2 \kappa^2 + 9 \psi (\psi - 1)).
\end{equation}
Both equations  impose   strict constraints on the       curvature $\kappa(s)$ for any Lancret curve   that lies on a circular cylinder.

This is in great contrast to the case of a Lancret curve lying on a sphere, where the ODE \eqref{es} can be explicitly integrated in terms of elementary functions   \cite{gray}.

\subsection{Curves with constant curvature}
A more intriguing question arises when the curvature is constant ($\kappa = \kappa_0$) but the torsion $\tau(s)$ is allowed to vary. Unlike the previous case, such a curve is not necessarily a circular helix. The following result establishes a remarkable necessary condition expressed only in terms of the torsion of the curve, just  as required in the initial problem. In other words, we are able to make the differential equation \eqref{e1} of Theorem \ref{t1} explicit.

\begin{theorem}\label{t2}
Let $\mathbf{r}(s)$ be a regular curve constrained to a cylinder of radius $\rho$. If the curvature is a non-zero constant $\kappa(s) = \kappa_0$, then one of the following holds: 
\begin{enumerate}
\item $\bfr$ is a circular helix.
\item The torsion is not constant and it  must satisfy the following first-order   differential equation:
\begin{equation}\label{eqs}
 \frac{\rho^4 \tau'^2}{9 - 4\rho^2\tau^2} = \rho^2\kappa_0^2 - \frac{1}{2} + \frac{\rho^2\tau^2}{9} \mp \frac{1}{6}\sqrt{9 - 4\rho^2\tau^2}. \end{equation}
In particular, it is necessary that $|\tau|\leq \frac{3}{2\rho}$.
\end{enumerate}
\end{theorem}

\begin{proof}
 Assuming $\kappa(s) = \kappa_0$, the polynomial equation \eqref{eq7} simplifies   to
\[
(\rho^2 \kappa_0^2 - \psi^2) (\rho^2 \tau^2 + 9 \psi (\psi - 1)) = 0.
\]
If the first factor vanishes, then $\psi = \rho \kappa_0$. This proves that $\psi$ is constant. \textcolor{blue}{By the definition of $\psi$, we have that $\alpha$ is constant. From \eqref{cog} and \eqref{cob}, we deduce $\cos\beta=0$ and $\cos\gamma=\pm\sin\alpha$. Finally, equation \eqref{ab33} becomes $\tau\psi\pm \kappa_0\sin\alpha\cos\alpha=0$. This proves that $\tau$ is constant because $\psi\not=0$.}

If the second factor vanishes, then 
$$\tau^2 = \frac{9}{\rho^2}\psi(1-\psi).$$
This gives 
\begin{equation}\label{pp1}
 \psi = \frac{1}{2} \pm \frac{1}{6}\sqrt{9 - 4\rho^2\tau^2} 
 \end{equation}
and, consequently, 
\begin{equation}\label{pp2}
 \psi' = \mp \frac{2\rho^2\tau\tau'}{3\sqrt{9 - 4\rho^2\tau^2}}. 
 \end{equation}
On the other hand, since $\kappa$ is constant, equation \eqref{p1} is 
$$ \psi' = \pm \frac{2}{3}\tau \sqrt{\rho^2\kappa_0^2 - \psi^2}. $$
Equating this expression with \eqref{pp2},  and assuming $\tau \neq 0$ (thus excluding planar curves), we cancel the common factor $\frac{2}{3}\tau$, to obtain
$$ \mp \frac{\rho^2\tau'}{\sqrt{9 - 4\rho^2\tau^2}} = \pm \sqrt{\rho^2\kappa_0^2 - \psi^2}. $$
Squaring both sides, and using the expression for $\psi$ in \eqref{pp1}, we get \eqref{eqs}. 
\end{proof}
 
 \begin{corollary}
Let $\tau(s)$ be the torsion of a curve lying on a cylinder of constant radius $\rho$, and let $\kappa_0$ be a non-zero constant. Then    the general solution $\tau(s)$  of \eqref{eqs}   is   given by the elliptic integral equation
\begin{equation}\label{eq26}
    \pm \frac{s - s_0}{3\rho} = \int_{u_0}^{\sqrt{9 - 4\rho^2\tau(s)^2}} \frac{dx}{\sqrt{(9 - x^2)\left(36\rho^2\kappa_0^2 - 9 \mp 6x - x^2\right)}},
\end{equation}
where $s_0$ is an initial arc length parameter and $u_0 = \sqrt{9 - 4\rho^2\tau(s_0)^2}$.
\end{corollary}

\begin{proof}
Introduce the change of variable $u(s) = \sqrt{9 - 4\rho^2\tau(s)^2}$, which implies $\rho^2\tau^2 = \frac{1}{4}(9 - u^2)$. Then  \eqref{eqs} becomes 
\begin{equation*}
    \frac{\rho^2 u'^2}{4(9 - u^2)} = \rho^2\kappa_0^2 - \frac{1}{2} + \frac{9 - u^2}{36} \mp \frac{u}{6}.
\end{equation*}
Simplifying the right-hand side and isolating $u'^2$, we obtain the following first-order   differential equation
\begin{equation*}
    u'^2 = \frac{9 - u^2}{9\rho^2} \left( 36\rho^2\kappa_0^2 - 9 \mp 6u - u^2 \right).
\end{equation*}
Taking the square root and separating the variables $u$ and $s$, we arrive at
\begin{equation*}
    \frac{du}{\sqrt{(9 - u^2)\left( 36\rho^2\kappa_0^2 - 9 \mp 6u - u^2 \right)}} = \pm \frac{ds}{3\rho}.
\end{equation*}
This proves the result.
\end{proof}
 Although \eqref{eq26} is difficult to solve in full generality, a particular case is of interest, namely, $\rho\kappa_0=1$.

\begin{corollary}\label{cor34}
Under the   condition $\kappa_0 = 1/\rho$, the   differential equation \eqref{eqs} for the torsion $\tau(s)$ admits an  explicit,   exact solution. Specifically, let   $Y(s) = C e^{\pm \frac{2\sqrt{2}}{\rho} s}$, where $C > 0$ is a constant of integration depending on the initial conditions. Then the torsion of the curve is given by
\begin{equation*}
    \tau(s) = \pm \frac{6\sqrt{2}}{\rho} \frac{\sqrt{Y(s)}\big|Y(s) - 1\big|}{Y(s)^2 + 6Y(s) + 1}.
\end{equation*}
\end{corollary}

\begin{proof}
Imposing the condition $\kappa_0 = 1/\rho$,   the integral in the right-hand side of \eqref{eq26} reduces to
\begin{equation*}
    I = \int \frac{dx}{(3-x)\sqrt{x^2 + 12x + 27}},
\end{equation*}
where $x=\sqrt{9-4\rho^2\tau^2}$ is restricted to the interval $[0,3]$. 
To evaluate this integral, we apply the   substitution $v = \frac{1}{3-x}$, which gives 
\begin{equation*}
    I = \int \frac{dv}{\sqrt{72v^2 - 18v + 1}} = \frac{1}{6\sqrt{2}} \ln \left| 144v - 18 + 12\sqrt{2}\sqrt{72v^2 - 18v + 1} \right|.
\end{equation*}
Reverting the substitution $v = \frac{1}{3-x}$, we obtain 
\begin{equation*}
    \pm \frac{s - s_0}{3\rho} = \frac{1}{6\sqrt{2}} \ln \left| \frac{144 - 18(3-x) + 12\sqrt{2}\sqrt{x^2+12x+27}}{3-x} \right| \Bigg|_{x=u_0}^{x=u(s)},
\end{equation*}
After some manipulations, we absorb the initial conditions into a single positive integration constant $C$. Then the 
solution takes the exponential form:
\begin{equation*}
       \frac{15 + 3x + 2\sqrt{2}\sqrt{x^2+12x+27}}{3-x}=C e^{\pm \frac{2\sqrt{2}}{\rho} s}=Y(s).
\end{equation*}
We can rearrange this expression to isolate the square root and square both sides, yielding a quadratic equation in $x$:
\begin{equation*}
    \left[ Y(s)(3-x) - 3(x+5) \right]^2 = 8(x^2 + 12x + 27).
\end{equation*} 
Simplifying, we obtain
\begin{equation*}
    (Y^2 + 6Y + 1)x^2 - 6(Y-1)^2 x + 9(Y^2 - 10Y + 1) = 0.
\end{equation*}
The discriminant $\Delta$ of this quadratic equation is  a perfect square, namely, $\Delta= (48Y)^2$. The roots of the equation are   $x = \frac{6(Y-1)^2 \pm 48Y}{2(Y^2 + 6Y + 1)}$. One of them is   $x = 3$,  which implies $\tau = 0$. This case is discarded. The other root is  
\begin{equation*}
    x(s) = 3 \frac{Y^2 - 10Y + 1}{Y^2 + 6Y + 1}.
\end{equation*}
Substituting this   form of $x(s)$ into the definition $\rho^2\tau^2 = \frac{1}{4}(9 - x^2)$, we conclude the proof.
\end{proof}

\section{Explicit examples of cylindrical curves}\label{s4}
In this final section, we show  the robustness of the characterization presented in Theorem \ref{t1} by applying it to some of the most famous curves  lying on circular cylinders.

\subsection{Planar curves (ellipse)}   
Suppose that $\bfr$ is an ellipse, for which we know its curvature $\kappa$ and its torsion, $\tau=0$. If $\kappa$ is constant, $\kappa(s)=\kappa_0>0$, then $\bfr$ is a circle of radius $1/\kappa_0$, hence it is included in a circular cylinder of radius $1/\kappa_0$. Suppose from now on that $\bfr$ is not a circle, that is, $\kappa$ is not constant. From \eqref{p1}, we deduce   
\[ \frac{\psi'}{\psi} = \frac{2}{3} \frac{\kappa'}{\kappa}. \]
Integrating yields
\begin{equation}\label{p5}
 \psi(s) = c   \kappa(s)^{2/3}, 
 \end{equation}
where $c$ is an integration constant. We calculate the value of $c$ by evaluating $\kappa$   at the vertex of the major axis  that has the greatest height, where   the tangent vector is perpendicular to the   axis of the cylinder,  meaning $\alpha = \pi/2$. This gives   $\psi = 1$.  Substituting this into \eqref{p5}, we find $ c = \kappa_{\max}^{-2/3}$. Then  \eqref{p5} becomes
\[ \psi(s) = \left( \frac{\kappa(s)}{\kappa_{\max}} \right)^{2/3}. \]

 By substituting \eqref{p5} into \eqref{eq7}, we obtain 
\[ \frac{c}{9} \kappa^{-4/3} \kappa'^2 = \frac{1}{\rho^2} (1 - c \kappa^{2/3}) (\rho^2\kappa^2 - c^2 \kappa^{4/3}). \]
We can explicitly recover the radius of the cylinder  by evaluating this equation at the minimum $\kappa_{\min}$ of $\kappa$.  At such a point,  the left-hand side is zero. In the   right-hand side,   the first factor vanishes when $\kappa_{min} = c^{-3/2} = \kappa_{\max}$. This implies that the initial curve is a circle. This case was discarded.  Thus,  the second factor must   vanish, yielding the radius of $\cy$: 
\begin{equation}\label{k5}
\rho = c \kappa_{\min}^{-1/3}=\kappa_{\max}^{-2/3}\kappa_{\min}^{-1/3}. 
\end{equation}
This  intrinsic deduction is in perfect agreement with classical geometry. Recalling that the extreme curvatures of an ellipse with major semi-axis $a$ and minor semi-axis $b$ are   $\kappa_{\max} = a/b^2$ and $\kappa_{\min} = b/a^2$, respectively, substituting these classical values into \eqref{k5}   yields $\rho = b$. This   confirms that the minor semi-axis of the elliptical section corresponds precisely to the radius of the   cylinder.

\subsection{Viviani's curve}

The Viviani curve is defined as the intersection of a sphere of radius $r=2\rho$ and a circular cylinder of radius $\rho$, where the axis of the cylinder  passes through the center of the sphere. This curve provides a sophisticated test case because its curvature and torsion    must simultaneously satisfy the spherical characterization \eqref{es} and our cylindrical constraint \eqref{eq7} as well as  the ODE \eqref{p1}.

First, we address the spherical condition.   Given the pair $(\kappa, \tau)$, the curve lies on a sphere of radius $r$ if and only if it satisfies \eqref{es}. Thus, the function
\begin{equation*}
  \frac{1}{\kappa^2} + \left( \frac{\kappa'}{\kappa^2 \tau} \right)^2  
\end{equation*}
must be constant. If this constant is $r^2$, with $r>0$,   we can write $\tau$ as 
$$
    \tau = \frac{\kappa'}{\kappa \sqrt{r^2 \kappa^2 - 1}}.
$$
On the other hand, if   the curve is included in a cylinder of radius $\rho=r/2$, then $\kappa$ and $\tau$ must satisfy the ODE \eqref{e1}. To demonstrate the robustness of the algebraic constraint and the identification algorithm, we can verify this for an explicit parametrization. We consider the standard parametrization of the Viviani curve, where $\rho$ denotes the radius of the generating cylinder,
\begin{equation*}
    \mathbf{r}(t) = \rho \left( 1 + \cos t, \, \sin t, \, 2 \sin(t/2) \right).
\end{equation*}
See \cite{gray}. While the prime denotes the derivative with respect to the arc-length $s$, the dot notation is used to represent derivatives with respect to $t$.  Since this parametrization is not by arc length, we define the speed factor 
\begin{equation*}
    v(t) = \|\dot{\mathbf{r}}(t)\|= \rho \sqrt{\frac{3 + \cos t}{2}}.
\end{equation*}
The curvature and torsion are 
$$
    \kappa(t) = \frac{\sqrt{13 + 3\cos t}}{\rho (3 + \cos t)^{3/2}}, \quad \tau(t) = \frac{6 \cos(t/2)}{\rho (13 + 3\cos t)}.
$$
Then
$$
    \kappa'(s) = \frac{\dot{\kappa}(t)}{v(t)} = \frac{3\sqrt{2} \sin t (7 + \cos t)}{2\rho^2 (3 + \cos t)^{3} \sqrt{13 + 3\cos t}}.
$$
It is immediate that   $\bfr$ is a spherical curve because $\kappa$ and $\tau$ satisfy   \eqref{es}. Moreover,   $r^2 =   2\rho$. This gives the   radius of the candidate cylinder, namely, $\rho$. Here we can proceed in two different ways. 

First, we can consider the ODE \eqref{p1}. After some computations, we arrive at 
\begin{equation} \label{31}
    \dot{\psi} = \frac{2\sin t (5 + \cos t)}{(3 + \cos t)(13 + 3\cos t)} \psi \pm \frac{2 \sqrt{1 + \cos t} \sqrt{3 + \cos t}}{13 + 3\cos t} \sqrt{ \frac{13 + 3\cos t}{(3 + \cos t)^3} - \psi^2 }.
\end{equation}
We introduce the change of variables 
$$u(t)=\psi(t)(3+\cos t).$$
After a long computation, equation   \eqref{31} becomes 
\begin{equation} \label{eq:simplified_ODE}
    (13 + 3\cos t)\dot{u} + u\sin t = \pm 2\sqrt{1 + \cos t} \sqrt{ 13 + 3\cos t - (3 + \cos t)u^2 }.
\end{equation} 
 Instead of attempting a direct integration, we propose the constant  ansatz    $u(t) = 2$. This function trivially satisfies \eqref{eq:simplified_ODE}.  Thus, for $u(t)=2$, we find
\begin{equation*}
    \psi(t) = \frac{2}{3 + \cos t}.
\end{equation*}
Note that we have obtained $\psi(t)$ from \eqref{p1}. Now we must verify its consistency with equation \eqref{eq7}. Substituting this expression of $\psi(t)$ into the algebraic constraint leads to a lengthy computation, which can be readily verified via symbolic computation software. This confirms that the identity \eqref{eq7} holds perfectly.

Alternatively,   one can determine $\psi(t)$ directly  by evaluating the projection of the tangent vector onto the vertical direction $(0,0,1)$. We have 
\begin{equation*}
    \cos^2 \alpha = \langle \bft, (0,0,1) \rangle^2 = \frac{\dot{z}(t)^2}{v(t)^2} = \frac{\cos^2(t/2)}{1 + \cos^2(t/2)},
\end{equation*}
which leads to 
\[     \psi(t) =  \frac{2}{3 + \cos t}.\]
  Again,     cumbersome and lengthy algebraic calculations verify that $\psi$ satisfies \eqref{p1} and \eqref{eq7}.  
 This  confirms that the Viviani curve perfectly couples both spherical and cylindrical constraints, validating our ODE as a robust framework for characterizing complex geometric intersections.

\section*{Acknowledgements}
The author is   grateful to the anonymous referee for their constructive suggestions, which helped to  improve the final version of the manuscript.  The author also thanks to Luiz C. B. Da Silva (Dundee) for his useful discussions in the preparation of this work. The author   has been partially supported by MINECO/MICINN/FEDER grant no. PID2023-150727NB-I00,  and by the ``Mar\'{\i}a de Maeztu'' Excellence Unit IMAG, reference CEX2020-001105- M, funded by MCINN/AEI/10.13039/ 501100011033/ CEX2020-001105-M.

\section*{Declaration of competing interest}
No conflict of interest exists in the submission of this manuscript, and manuscript is approved by all authors for publication. No
funding was provided for this project. No data was used for the research described in the article.

\end{document}